\documentclass[11pt]{article}

\usepackage{pgf,tikz}
\usepackage{mathrsfs}
\usetikzlibrary{arrows}

\usepackage{pgfplots}
\pgfplotsset{compat=newest}

\usepackage{graphicx} 
\graphicspath{{figures/}} 

\usepackage{fullpage, url,amsmath,amsfonts,amssymb,mathtools,mathrsfs,graphicx,  algorithm, float, sansmath,epstopdf,color,caption,enumitem,tabularx}
\usepackage[final]{pdfpages}
\usepackage{amsthm}
\usetikzlibrary{automata,topaths}
\usetikzlibrary{decorations.pathreplacing,shapes.misc}
\usepackage{fancyhdr}

\def\beq{\begin{equation}}
\def\eeq{\end{equation}}
\def\baq{\begin{eqnarray}}
\def\eaq{\end{eqnarray}}
\def\baqn{\begin{eqnarray*}}
\def\eaqn{\end{eqnarray*}}

\usepackage{multirow}
\usetikzlibrary{calc,arrows}
\theoremstyle{plain}
\newtheorem{definition}{Definition}
\newtheorem{remark}{Remark}

\newtheorem{example}{Example}
\newtheorem{theorem}{Theorem}
\newtheorem{lemma}[theorem]{Lemma}

\usepackage{blindtext}

\usepackage[colorlinks,linkcolor=blue,citecolor=red]{hyperref}

\usepackage{xcolor, framed}

\newcommand{\R}{{\mathbb R}}

\newcommand{\interior}{{\rm int}\kern 0.06em}

\newcommand{\Id}{\mathrm{Id}}

\def\<{\langle}
\def\>{\rangle}

\usepackage{ulem}

\usepackage{subcaption}

\usepackage{pict2e}

\if
{

\usepackage[pageref]{backref}
\renewcommand*{\backrefalt}[4]{%
\ifcase #1 %
(Not cited)%
\or
(Cited on p.~#2)%
\else
(Cited on pp.~#2)%
\fi
}

}
\fi

\begin{document}
\title{{State-Dependent Sweeping Processes: Asymptotic Behavior and Algorithmic Approaches}}
\author{Samir Adly\thanks{Laboratoire XLIM, Universit\'e de Limoges,
123 Avenue Albert Thomas,
87060 Limoges CEDEX, France\vskip 0mm
Email: \texttt{samir.adly@unilim.fr}} \qquad
 Monica G. Cojocaru\thanks{Department of Mathematics and Statistics, University of Guelph, Guelph, ON N1G 2W1, Canada \vskip 0mm
Email: mcojocar@uoguelph.ca}
 \qquad Ba Khiet Le \thanks{Optimization Research Group, Faculty of Mathematics and Statistics, Ton Duc Thang University, Ho Chi Minh City, Vietnam\vskip 0mm
 E-mail: \texttt{lebakhiet@tdtu.edu.vn}}
 }
\date{}
\maketitle

\begin{abstract}
In this paper, we investigate the asymptotic properties of a particular class of state-dependent sweeping processes. While extensive research has been conducted on the existence and uniqueness of solutions for sweeping processes, there is a scarcity of studies addressing their behavior in the limit of large time. Additionally, we introduce novel algorithms designed for the resolution of quasi-variational inequalities.
As a result, we introduce a new derivative-free algorithm to find zeros of nonsmooth Lipschitz continuous mappings with a linear convergence rate. This algorithm can be effectively used in nonsmooth and nonconvex optimization problems that do not possess necessarily second-order differentiability conditions of the data.
\end{abstract}



\section{Introduction}
In the seventies, J. J. Moreau  \cite{M2,M3,M4}  introduces  and rigorously analyzed sweeping processes which have the form as follows
\begin{equation}\nonumber
\left\{
\begin{array}{l}
\dot{x}(t) \in -{\rm N}_{C(t)}(x(t))\; {\rm a.e.} \; t \in [0,T],\\ \\
x(0) = x_0\in C(0),
\end{array}\right.
\end{equation}
where ${\rm N}_{C(t)}(\cdot)$ denotes the normal cone operator of the convex set $C(t)$  in a Hilbert space $H$.  Originally, sweeping processes describe the movement of a particle $x(t)$ which belongs to a moving set $C(t)$. When the particle touches the boundary, it is forced to come back inside the moving set.  
{Sweeping processes find applications in a wide range of fields, including nonsmooth mechanics and non-regular electrical circuits. The well-posedness of such systems and their extensions has undergone extensive investigation (cf. \cite{aht,ahl,Kunze,t1,t2}, among others). However, research concerning the long-term behavior of trajectories, particularly as time approaches infinity, remains relatively limited (see, e.g., \cite{Daniilidis,le}).}


{In this paper, our focus lies in the study of the asymptotic behavior of the state-dependent sweeping process described as follows:
\beq\label{main}
\left\{
\begin{array}{l}
 {  \dot{x}(t) \in  -f(x(t))-{\rm N}_{K(x(t))}(x(t))}\; {\rm a.e.} \; t \in [0,T],\\ \\
   x(0)=x_0,
\end{array}\right.
\eeq
and its connection with Quasi-variational inequalities (QVIs). Here, the perturbation term $f: H\to H$ is assumed to be $L$-Lipschitz continuous, while the state-dependent moving set $K(x)=C+v(x)$ is defined such that $C$ represents a non-empty closed convex subset of $H$, and $v: H\to H$ is a $l$-Lipschitz continuous map. The well-posedness of equation (\ref{main}) has been well-established  if {$l<1$} (see, e.g., \cite{Kunze}). Our primary contribution in this work is to establish some mild condition under which the trajectory of (\ref{main}) converges to the unique equilibrium  $x^*\in K(x^*)$ with an exponential rate. This equilibrium point $x^*$ is a solution to a well-established variational inclusion, known in the literature as a Quasi-Variational Inequality (QVI), which can be  expressed as follows:
\beq\label{qvi}
0\in f(x^*)+{\rm N}_{K(x^*)}(x^*).
\eeq}
\noindent {Moreover, it's worth noting that the velocity also exhibits  exponential convergence towards zero. The QVI described in (\ref{qvi}) holds significant relevance and finds numerous applications in diverse fields, including engineering, economics, finance, transportation science, traffic management, ecology, energy systems, biomedical engineering, robotics, game theory, among others. For further exploration, we refer to works such as those by \cite{Bensoussan}, \cite{Bliemer}, and \cite{Luca}} etc. {Using a semi-implicit discretization scheme of (\ref{main}), we have 
\beq
\frac{x_{n+1}-x_n}{h}\in-f(x_n)-{\rm N}_{K(x_n)}(x_{n+1}),
\eeq
which further leads to the following formulation
\beq\label{cua}
x_{n+1}={\rm proj}_{K(x_n)}(x_n-hf(x_n)),
\eeq
{where $h>0$ represents the step size in the discretization scheme.}
}
{This method is a gradient type algorithm, as discussed in \cite{Nesterov}, and it is also commonly referred to as the well-known ``catching-up'' algorithm, as documented in \cite{Kunze} and \cite{M4}. It's worth noting that the catching-up algorithm exhibits slow convergence (see, e.g, \cite{Nesterov}), especially when subjected to the relatively restrictive condition $0<l < \frac{1}{\gamma (\gamma+\sqrt{\gamma^2-1})}$, where $f$ is supposed to be $\mu$-strongly monotone and $\gamma=\frac{L}{\mu}\ge 1$.}
{In \cite{Nesterov}, Nesterov-Scrimali significantly improved the convergence condition to $l < \frac{1}{\gamma}$ by solving a multitude of strongly monotone variational inequalities (VIs). More recently, under the same assumptions as in \cite{Nesterov}, Adly and Le \cite{al2} transformed the QVI into a VI, simplifying the problem to solving only a strongly monotone VI using the Douglas-Rachford splitting algorithm. In this paper, we propose a modification to the catching-up algorithm as follows:
\begin{equation}\label{mcua}
x_{n+1}={(\Id -v)^{-1}\Big({\rm proj}_{C}(x_n-v(x_n)-hf(x_n))\Big)}=(\Id-v)^{-1}\text{proj}_{K(x_n)}(x_n - hf(x_n)),
\end{equation}
aiming to achieve linear convergence under remarkably weaker conditions  than those outlined in \cite{al2} and \cite{Nesterov}, by considering the strong monotonicity of the pair $(f, \Id-v)$. It's important to note that our modified algorithm (\ref{mcua}) simplifies each iteration by requiring only one projection, eliminating the need to compute the resolvent of $f\circ (\Id-v)^{-1}$ as is done in \cite{al2}. While retaining the fundamental concepts of the original {catching-up algorithm} (\ref{cua}), our modification uses the inverse operator $(\Id-v)^{-1}$ to improve convergence rate with less stringent conditions. Indeed, the {catching-up} algorithm does not capture well the nature of QVIs as the algorithm (\ref{mcua}) does since the moving set depends on the state. Note that the proposed approach extends its applicability beyond scenarios where the strong monotonicity of $(f, \Id-v)$ is preserved, allowing a more general treatment where only the monotonicity of $(f, \Id-v)$, characterized by the absence of an obtuse angle between $f$ and $\Id-v$, is sufficient. Moreover, we even relax this condition by assuming the pseudo-monotonicity of $(f, \Id-v)$ where {the monotonicity} could be not observed.} As a consequence, we provide a new derivative-free to find a zero of a nonsmooth Lipschitz continuous mapping. 

{The paper is organized as follows. In Section 2, we revisit classical notions and provide essential definitions. In Section 3, we explore the asymptotic behavior of state-dependent sweeping processes. In Section 4, we introduce novel algorithms designed for solving  QVIs, covering both scenarios of $(f, \Id-v)$: strongly (pseudo) monotone and solely (pseudo) monotone cases. An interesting application of previous parts is presented in Section 5. Section 6 includes a collection of illustrated numerical examples. The paper concludes with a summary and some potential future directions of research.}

\section{Notation and Mathematical Backgrounds}
{In this paper, the framework is a real Hilbert space denoted as $H$, where the inner product  and the associated norm are respectively denoted by $\langle \cdot, \cdot \rangle$ and $\Vert \cdot \Vert$. Given a closed convex subset $C$ of $H$ and a point $x$ within $C$, we define the normal cone to $C$ at $x$ as follows:
\[
{\rm N}_C(x) = \{x^* \in H: \langle x^*, y-x \rangle \leq 0, \text{ for all } y \in C\}.
\]
Additionally, we define the projection from a point $z$ onto $C$ as:
\[
\text{proj}_C(z) = \{x \in C \text{ such that } d(z,C) = \|z - x\|\}.
\]
Here, $d(z,C)$ represents the distance from point $z$ to set $C$.
}


\begin{definition}
\noindent  A  mapping  $A: H\to H$ is called $L$-Lipschitz continuous ($L>0$) if
$$
\Vert Ax-Ay\Vert \le L \Vert x-y \Vert\;\;\forall\;x, y\in H.
$$
It is called contractive if it is $L$-Lipschitz continuous with $L<1$.\\

\noindent  It is called  $\mu$-strongly monotone ($\mu>0$) if
$$
\langle  Ax-Ay, x-y \rangle \ge \mu\Vert x-y \Vert^2\;\;\forall\;x, y\in H.
$$
\end{definition}
Similarly  let us recall about the monotonicity of set-valued operators. 
\begin{definition}
\noindent {A set-valued mapping $F: {H}\rightrightarrows {H}$  is called monotone provided
$$
\langle x^*-y^*,x-y \rangle \ge 0 \;\;\forall \;x, y\in H, x^*\in F(x) \;{\rm and}\;y^*\in F(y).
$$
In addition, it is called maximal monotone if there is no monotone operator $G$ such that the graph of $F$ is strictly {included} in  the graph of $G$. }
\end{definition}

\noindent {The resolvent  of $F$ is defined respectively by
$
J_B:=(Id+B)^{-1}
$}. It is easy to see that the resolvent is non-expansive. \\

Finally we introduce the monotonicity of a pair of functions, which plays an important role in the study of QVIs. 
\begin{definition}\label{pair-monot}
\begin{itemize}
 \item The pair of {maps} $(A_1, A_2)$ is called  monotone if 
\beq\label{pair_mon}
\langle  A_1x-A_1y, A_2x-A_2y \rangle \ge  0\;\;\forall\;x, y\in H.
\eeq
It means that the increments of $A_1$ and $A_2$ do not form an obtuse angle.\\

\item The pair $(A_1, A_2)$ is called $\mu$-strongly monotone if 
$$
\langle  A_1x-A_1y, A_2x-A_2y \rangle \ge  \mu \Vert x-y \Vert^2\;\;\forall\;x, y\in H.
$$
\item $(A_1, A_2)$ is called pseudo-monotone if 
$$
\langle  A_1x, A_2y-A_2x \rangle \ge 0 \Rightarrow \langle  A_1y, A_2y-A_2x \rangle \ge 0\;\;\forall\;x, y\in H.
$$
It is easy to see that if $(A_1, A_2)$ is monotone then it is pseudo-monotone.\\

\item $(A_1, A_2)$ is called $\mu$-strongly pseudo-monotone if 
$$
\langle  A_1x, A_2y-A_2x \rangle \ge 0 \Rightarrow \langle  A_1y, A_2y-A_2x \rangle \ge  \mu \Vert x-y \Vert^2\;\;\forall\;x, y\in H.
$$
\end{itemize}
\end{definition}
\begin{remark}\normalfont {
(i) We observe that if one of the two operators in Definition \ref{pair-monot} is the identity operator, the definition in \eqref{pair_mon} simplifies to the standard monotonicity for the remaining operator.\\
(ii) If the two operators are linear, then  $(A_1, A_2)$-monotonicity is equivalent to
$$
\langle  A_1x, A_2x \rangle \ge  0\;\;\forall\;x\in H.
$$
For example the following pair of matrices are $2$-strongly monotone (and hence monotone)
$$
A_1 = \begin{bmatrix}
2 & 1 \\
1 & 2
\end{bmatrix}, 
\quad
A_2 = \begin{bmatrix}
3 & 1 \\
1 & 3
\end{bmatrix}.
$$
(iii) Consider the nonlinear operators $A_1,\,A_2:\R^2\to\R^2$ defined by
\[
A_1(x_1, x_2) = (-x_1^2, 0) \mbox{ and } A_2(x_1, x_2) = (0, x_2^2).
\]
We have, for all \(x, y \in \mathbb{R}^2\), 
\[
\langle A_1x - A_1y, A_2x - A_2y \rangle = 0,
\]
which proves that the pair $(A_1,A_2)$ is monotone but not strongly monotone.
}

\end{remark}

\section{Asymptotic Behavior of State-Dependent Sweeping Processes}
{We start this section by outlining the foundational assumptions that will guide our exploration into the asymptotic behavior of state-dependent sweeping processes.}
%
\vskip 2mm
\noindent {\textbf{Assumption 1}: Let \(K(x) = C + v(x)\), where \(C\) is a non-empty, closed, and convex subset of \(H\) and \(v: H \to H\) is  \(l\)-Lipschitz continuous. In addition,    the inverse operator \((\Id  - v)^{-1}\) is well-defined and  \(\tilde{l}\)-Lipschitz continuous.} 
\vskip 2mm
\noindent \textbf{Assumption 2}: $f:H\to H$  is $L$-Lipschitz continuous.\\
\vskip 2mm
\noindent \textbf{Assumption 3}: The pair $(f, \Id -v)$ is $\gamma$-strongly monotone.
\begin{remark}\normalfont
If $f$  is $\mu$-strongly monotone and $l\le\mu/L$ as usually used in the literature (see, e.g., \cite{al2,Nesterov}), then Assumptions 1, 3 are satisfied with $\tilde{l}=\frac{1}{1-l}$ and $\gamma=\mu-lL$. Indeed, for all $x, y\in H$, one has
$$
\Vert (\Id -v)^{-1}x-(\Id -v)^{-1}y\Vert\le \frac{1}{1-l} \Vert x-y \Vert
$$
and
$$
\langle f(x)-f(y), x-y-v(x)+v(y) \rangle\ge (\mu-lL) \Vert x-y\Vert^2.
$$
\end{remark}
The following result shows that the monotonicity of the pair  $(f, \Id -v)$ is equivalent to the monotonicity of $f\circ (\Id -v)^{-1}$.
{
\begin{lemma}\label{lemmon}
\def\labelenumi{{\rm (\roman{enumi})}}
\begin{enumerate}
    \item $f\circ(\Id -v)^{-1}$ is (pseudo) monotone if and only if the pair $(f, \Id -v)$ is (pseudo) monotone. Furthermore, if the pair $(f, \Id -v)$ is $\gamma$-strongly monotone, then $f\circ(\Id -v)^{-1}$ is $\frac{\gamma}{(1+l)^2}$-strongly monotone.
    \item If $v$ is a linear operator, then   $(f, \Id -v)$ is (strongly) monotone $\Leftrightarrow$  $(\Id -v)^T\circ f$ is (strongly) monotone.
    \item If $f$ is a linear operator, then    $(f, \Id -v)$ is (strongly) monotone $\Leftrightarrow$ $f^T\circ (\Id -v)$ is (strongly) monotone.
\end{enumerate}
\end{lemma}
}

\begin{proof}
(i) Let $x_i \in H$ and $y_i=(\Id -v) (x_i), \;i=1,2$. We have 
\baqn
\langle f(x_1)-f(x_2), (\Id -v) (x_1-x_2)\rangle
= \langle f(\Id -v)^{-1}(y_1)-f(\Id -v)^{-1}(y_2), y_1-y_2\rangle.
\eaqn
Thus $f\circ(\Id -v)^{-1}$ is (pseudo) monotone if and only if the pair $(f, \Id -v)$ is (pseudo) monotone.
Since $\Vert y_1-y_2 \Vert \le (1+l)\Vert x_1-x_2 \Vert$, if $(f, \Id -v)$ is $\gamma$-strongly monotone, then $f\circ(\Id -v)^{-1}$ is $\frac{\gamma}{(1+l)^2}$-strongly monotone. \\
(ii) If $v$ is linear then
$$
\langle (\Id -v)^T(f(x_1)-f(x_2)),x_1-x_2\rangle=\langle f(x_1)-f(x_2), (\Id -v) (x_1-x_2)\rangle
$$
and the remain follows.\\
(iii) Similar to the case where $f$ is linear. 
\end{proof}

\begin{remark}\normalfont
{In practice, we often encounter scenarios where  $v$ or $f$ is a large matrice, making it convenient to verify the (strong) monotonicity of $(f,\Id -v)$}.
\end{remark}
\noindent Note that under the setting of Assumption 1,  (\ref{main}) can be rewritten as follows
\beq\label{main1}
\left\{
\begin{array}{l}
   \dot{x}(t) \in  -f(x(t))-{\rm N}_{C}(x(t)-v(x(t)))\\ \\
   x(0)=x_0.
\end{array}\right.
\eeq
\begin{lemma}\label{lem1}
Let Assumptions 1, 2, 3 hold. Then the system  (\ref{main1})  has exactly an equilibrium point $x^*$, i.e., satisfying $0\in f(x^*)+{\rm N}_{K(x^*)}(x^*)$, with $K(x^*)=C+v(x^*)$.

\end{lemma}
\begin{proof}
 We have
$$
0\in f(x^*)+{\rm N}_{K(x^*)}(x^*)= f(x^*)+{\rm N}_C(x^*-v(x^*))=f\circ (\Id -v)^{-1}(y^*)+{\rm N}_C(y^*)
$$
where $y^*=(\Id -v)(x^*)$. The mapping $f\circ (\Id -v)^{-1}$ is strongly monotone (Lemma \ref{lemmon}) and Lipschitz continuous. Thus $\Big(f\circ(\Id -v)^{-1}+{\rm N}_C\Big)$ is a  strongly maximal monotone  operator. Hence $y^*=(f\circ(\Id -v)^{-1}+{\rm N}_C)^{-1}(0)$  exists uniquely and so is $x^*$.
\end{proof}

{The subsequent result outlines a condition on $v$ which ensures that both the trajectory $t\mapsto x(t)$ of the SWP (\ref{main}) and its velocity $\dot{x}(\cdot)$ converge (the latter to  zero) at an exponential rate}.
\begin{theorem}
Let Assumptions 1, 2, 3 hold and suppose that $v$ is linear, symmetric and {contractive}. Then 
$$
\Vert x(t)-x^*\Vert\le \beta e^{-\alpha t}\Vert x_0-x^*\Vert \;\;\forall t\ge 0
$$
for some $\alpha>0, \beta>0$. In addition, we have
$$
\Vert \dot{x}(t)\Vert\le Ce^{-\alpha t}\Vert \dot{x}(t_0)\Vert \;\;a.e. \,t\ge t_0
$$
for some $C>0$ where $t_0$ is the first time that $x$ is differentiable. 
\end{theorem}
\begin{proof} Define $V(t):=\langle Bx(t)-Bx^*,x(t)-x^*\rangle$ where $B=(\Id -v)/2$ is a symmetric, positive definite linear operator.  Obviously 
$$
\frac{1-l}{2}   \Vert x(t)-x^*\Vert^2\le V(x(t)) \le  \Vert x(t)-x^*\Vert^2.
$$
We have 
$$
\dot{{x}}(t) +f(x(t))\in  -{\rm N}_{C}(x(t)-v(x(t))) \;\;{\rm and }\; f(x^*) \in  -{\rm N}_C(x^*-v(x^*)).
$$
The monotonicity of ${\rm N}_C$ implies that 
\beq
\langle \dot{{x}}(t) +f(x(t))-f(x^*), x(t)-x^*-v(x(t))-v(x^*)\rangle \le 0.
\eeq
From the $\gamma$-strong monotonicity of $(f, \Id -v)$, one has 
\baqn
\frac{dV}{dt}(t)&=&\langle \dot{{x}}(t), 2B(x(t)-x^*)\rangle\le - \langle f(x(t))-f(x^*), x(t)-x^*-v(x(t))-v(x^*)\rangle \\
&\le & -\gamma \Vert x(t)-x^*\Vert^2\le -{\gamma} V(x(t)).
\eaqn
Using Gronwall's inequality, one has
$$
\frac{1-l}{2} \Vert x(t)-x^*\Vert^2\le V(x(t))\le e^{-{\gamma t}}V(x_0)\le e^{-{\gamma t}}\Vert x_0-x^*\Vert^2
$$
and the conclusion follows with $\alpha={\gamma}/{2}$ and $\beta=\sqrt{2/(1-l)}$. \\

Since the solution $x$ is absolutely continuous, it is differentiable almost everywhere. Similarly, we can consider the function $W(t)=\langle Bx(t+\epsilon)-Bx(t),x(t+\epsilon)-x(t)\rangle$ for some fixed $\epsilon>0$ and obtain that 
$$
\Vert x(t+\epsilon)-x(t)\Vert\le \beta e^{-\alpha t}\Vert x(t_0+\epsilon)-x(t_0)\Vert 
$$
which is equivalent to
$$
\Vert \frac{x(t+\epsilon)-x(t)}{\epsilon}\Vert\le \beta e^{-\alpha t}\Vert \frac{x(t_0+\epsilon)-x(t_0)}{\epsilon}\Vert.
$$
One has the conclusion by letting $\epsilon\to 0$. 
\end{proof}
\begin{remark}\normalfont
{To achieve exponential convergence for the trajectory of (\ref{main}), one can employ the technique described in \cite{al2}. This approach becomes effective when the Lipschitz constant, denoted as $l$, of the function $v$ satisfies the condition $l < \frac{\mu}{L}$, where $f$ exhibits $\mu$-strong monotonicity. However, it's worth noting that in such cases, $l$ tends to be very small.\\
Fortunately, with the introduction of the new concept of monotonicity for the pair $(f, v)$, we can now work with a much more lenient requirement: $l < 1$. This condition appears to be optimal, as having $l \geq 1$ can potentially disrupt the dissipation of the system, making it challenging to find a suitable Lyapunov function.\\
It's important to mention that when dealing with algorithms designed to solve the QVI (\ref{qvi}), the need for $l < 1$ is even further relaxed, as there's no requirement to find a Lyapunov function in this context.
}
\end{remark}
\section{New  Algorithms for Solving QVIs}
\normalfont {It is easy to see that if $x^*\in K(x^*)$ is an equilibrium point of \eqref{main1}, then $y^*=x^*-v(x^*)$ is a solution of the following variational inequality
\beq\label{vire}
{\rm VI}(T,C)\left\{
\begin{array}{l}
\mbox{Find }y\in C \mbox{ such that }\\
0\in T(y)+{\rm N}_C(y),
\end{array}
\right.
\eeq
with $T=f\circ (\Id-v)^{-1}$. Conversely, if $y^*\in C$ is a solution of ${\rm VI}(T,C)$, then $x^*=(\Id-v)^{-1}(y^*)$ is an equilibrium point of \eqref{main1}.
} Therefore, to solve the QVI  (\ref{qvi}), the strong monotonicity or even only monotonicity of $f$ is not required. All we need is the monotonicity of $f(\Id -v)^{-1}$, or equivalently the monotonicity of $(f, \Id -v)$.
\begin{remark}\label{rem5}\normalfont
In one dimension we let $f(x)=-x+\frac{1}{3}\sin x$, $C=\R^+$ and $v(x)=2x+\frac{1}{3}\cos x$. Clearly, $f$ is not monotone. However 
\baqn
&&\langle f(x_1)-f(x_2), (\Id -v) (x_1-x_2)\rangle\\
&=&\langle -(x_1-x_2)+\frac{1}{3}(\sin x_1-\sin x_2),  -(x_1-x_2)-\frac{1}{3}(\cos x_1-\cos x_2)\rangle\\
&\ge& (1-\frac{1}{3}-\frac{1}{3}-\frac{1}{9})\Vert x_1-x_2\Vert^2\ge \frac{2}{9} \Vert x_1-x_2\Vert^2.
\eaqn
 Thus  $(f,\Id -v)$ is strongly monotone and the QVI (\ref{qvi}) is still solvable.  
 
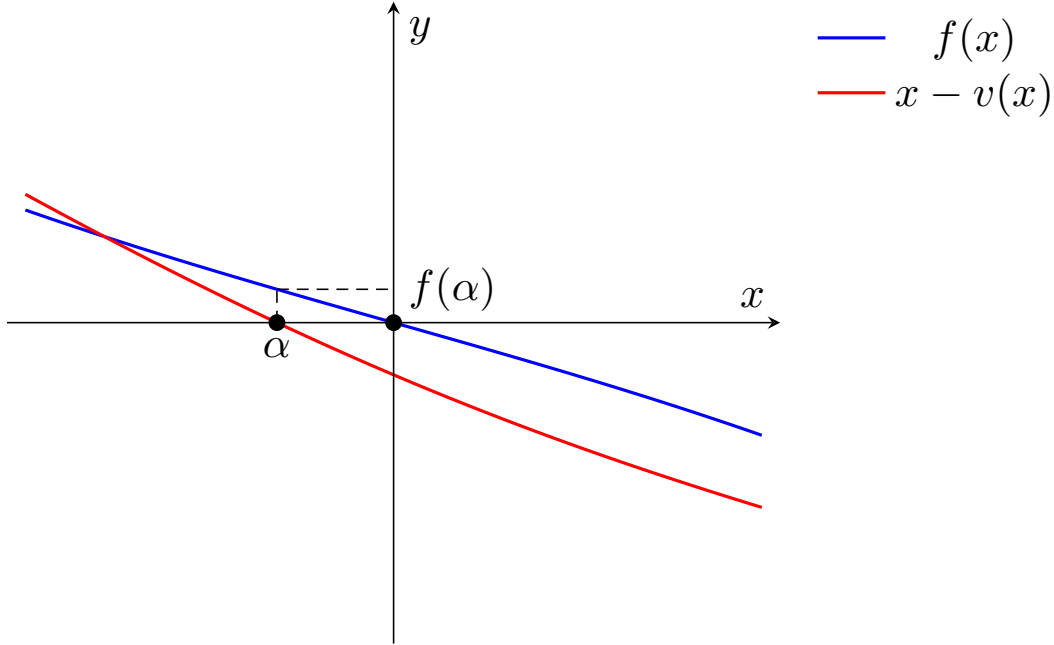
\begin{figure}
\begin{center}
 \begin{tikzpicture}[scale=1.5]
\begin{axis}[
    xlabel={$x$},
    ylabel={$y$},
    xmin=-1, xmax=1,
    ymin=-2, ymax=2,
    axis lines=center,
    axis on top=true,
    domain=-1:1,
    samples=100,
    ticks=none, 
    enlargelimits={abs=0.05},
    legend pos=outer north east,
    legend style={draw=none},
]

\addplot[blue, thick]{-x + (1/3) * sin(deg(x))};
\addlegendentry{$f(x)$}

\addplot[red, thick]{x - (2*x + (1/3) * cos(deg(x)))};
\addlegendentry{$x - v(x)$}

\node[shape=circle,fill=black,inner sep=1.5pt] at (axis cs:-0.3168,0) {};
\node[below] at (axis cs:-0.3168,0) {$\alpha$};

\node[shape=circle,fill=black,inner sep=1.5pt] at (axis cs:0,0) {};

\pgfmathsetmacro{\fOfAlpha}{-(-0.3168) + (1/3) * sin(deg(-0.3168))}

\draw[densely dashed, black] (axis cs:-0.3168,0) -- (axis cs:-0.3168,\fOfAlpha);
\draw[densely dashed, black] (axis cs:-0.3168,\fOfAlpha) -- (axis cs:0,\fOfAlpha);

\pgfmathsetmacro{\fAlpha}{-(-0.3168) + (1/3)*sin(deg(-0.3168))}

\node[right] at (axis cs:0,\fAlpha) {$f(\alpha)$};

\end{axis}
\end{tikzpicture}
\caption{Plot of the functions  $f(x)$ and $x - v(x)$}
\label{alpha}
\end{center}
\end{figure}
\noindent
{The only solution for the QVI (\ref{qvi}) is found to be the fixed point of the function $v$, which we call $\alpha$ and is approximately $-0.3168$ (see Figure \ref{alpha}). Basically, to solve the QVI (\ref{qvi}), we look for a solution to the problem $0\in f(x)+{\rm N}_{\mathbb{R}^+}(x-v(x))$. This can be rewritten equivalently as
\begin{equation}\label{qvimax}
\max\big(x-v(x)-f(x), 0\big) = x-v(x).
\end{equation}
If $x-v(x)-f(x) \geq 0$, then equation \eqref{qvimax} implies that $f(x) = 0$, which would make $x = 0$. But $x = 0$ doesn't work because $0 - v(0) - f(0) = -\frac{1}{3} < 0$. If, however, $x-v(x)-f(x) \leq 0$, then \eqref{qvimax} leads us to $v(x) = x$. This means that our solution is $x = \alpha$, since for $\alpha$ the condition $\alpha - v(\alpha) - f(\alpha) < 0$ is true (because $f(\alpha) > 0$).\\
The continuous trajectory $t\mapsto x(t)$, $t\in [0,T]$ of \eqref{main} is plotted in Figure \ref{swp_rem5}. We observe that the trajectory converges to $\alpha$.
}

\begin{figure}
\begin{center}
\includegraphics[width=4in]{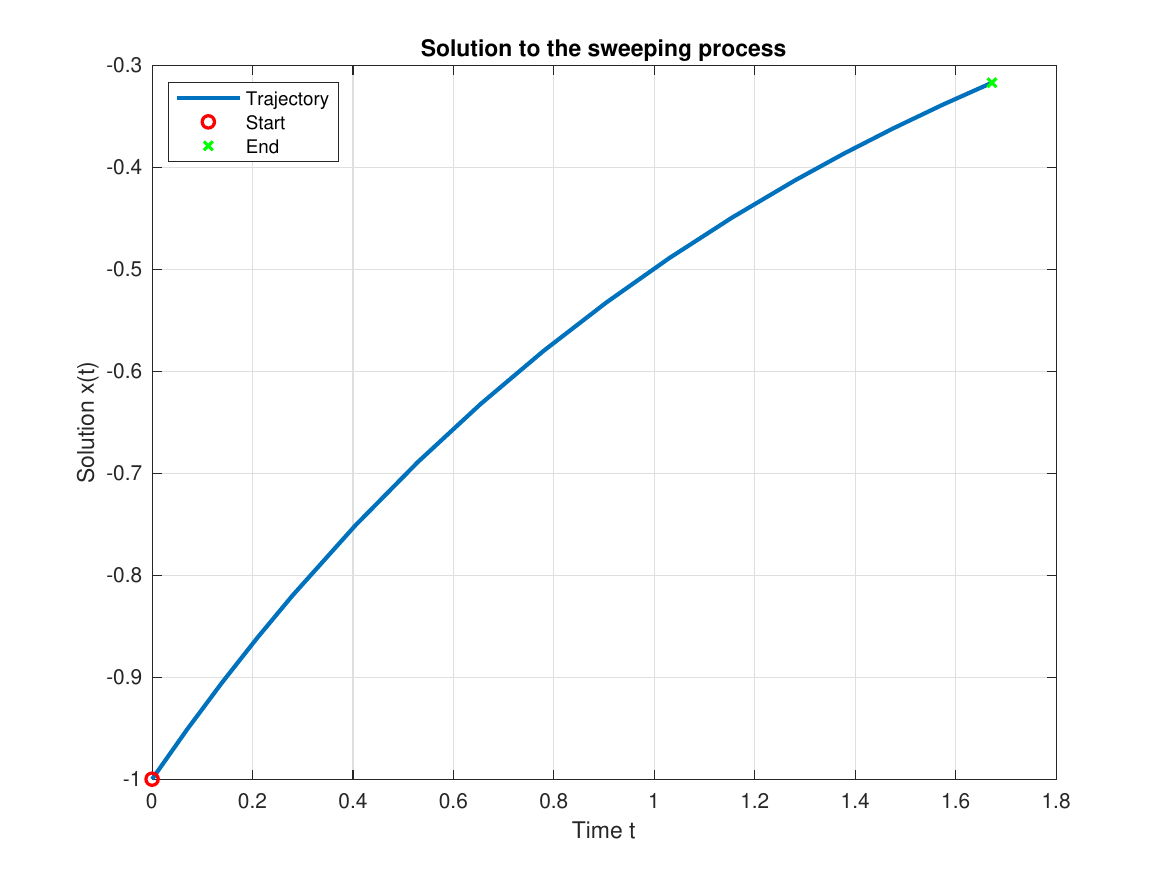}
\caption{Trajectory solution $t\mapsto x(t)$  to the sweeping process  \eqref{main} with the data of Remark \ref{rem5}.} 
\label{swp_rem5}
\end{center}
\end{figure}

\end{remark}
\subsection{The strong monotonicity case}
Under the strong monotonicity of $(f, \Id -v)$, we provide an algorithm to solve the QVI  (\ref{qvi}) with linear convergence rate as follows. \\

\noindent \textbf{Algorithm \;1}: Given $x_0\in H$, for $n\ge 0$, we compute 
\beq
x_{n+1}=(\Id -v)^{-1}\Big({\rm proj}_{C}(x_n-v(x_n)-hf(x_n))\Big)=(\Id -v)^{-1}\Big({\rm proj}_{K(x_n)}(x_n-hf(x_n))\Big).
\eeq

\begin{theorem}\label{ro1}
Suppose that  Assumptions 1, 2, 3 hold. Let $(x_n)$ be the sequence generated by Algorithm 1 with $h:=\frac{\alpha }{L^2 \tilde{l}^2}$. Then we have 
$$
\Vert x_n-x^*\Vert \le {\tilde{l}(1+l)} \kappa^n\Vert x_0-x^*\Vert,
$$
where $\kappa:=\sqrt{1-\frac{\alpha^2}{L^2\tilde{l}^2}}$, $\alpha:=\frac{\gamma}{(1+l)^2}$ and $x^*$ is the unique solution of the QVI (\ref{qvi}).
\end{theorem}
\begin{proof}
Let $y_n:=(\Id -v)x_n$ and $g:=f(\Id -v)^{-1}$ then we have $y_{n+1}={\rm proj}_{C}(y_n-hg(y_n))$ and $g$ is $\alpha$-strongly monotone and ${L\tilde{l}}$- Lipschitz continuous (Lemma  \ref{lemmon}). On the other hand, let $y^*=(\Id -v)x^*$, we have $0\in g(y^*)+{\rm N}_C(y^*) \Leftrightarrow y^*={\rm proj}_{C}(y^*-hg(y^*))$. Since the projection operator is non-expansive and $g$ is $\alpha$-strongly monotone, we have 
\baqn
\Vert y_{n+1}-y^*\Vert^2&\le& \Vert y_n-hg(y_n)-y^*+hg(y^*)\Vert^2\\
&\le&\Vert y_n-y^*\Vert^2(1-2h\alpha+h^2{L^2\tilde{l}^2}).
\eaqn
By choosing $h:=\frac{\alpha }{L^2 \tilde{l}^2}$, we have 
\beq
\Vert y_{n+1}-y^*\Vert^2\le \Vert y_n-y^*\Vert^2(1-\frac{\alpha^2}{L^2\tilde{l}^2}),
\eeq
which deduces that 
$$
\Vert y_n-y^*\Vert\le {\kappa^n} \Vert y_0-y^*\Vert.
$$
Note that $\Vert x_n-x^*\Vert = \Vert (\Id -v)^{-1} (y_n-y^*) \Vert \le \tilde{l}\Vert y_n-y^*\Vert$ and $\Vert y_0-y^*\Vert = \Vert x_0-x^*-(v(x_0)-v(x^*))\Vert \le (1+l)\Vert x_0-x^*\Vert$ and  the conclusion follows. 
\end{proof}

\begin{remark}\normalfont
  Algorithm 1 keeps the basic of the catching-up algorithm with the additional operator $(\Id -v)^{-1}$. The inverse  operator $(\Id -v)^{-1}$ captures well the nature of the QVI  (\ref{qvi}) and thus allows the algorithm to obtain the linear convergence rate under very general conditions. Indeed our conditions are weaker than classical conditions (see, e.g., \cite{al2,Nesterov}) and thus weaker than the conditions used by  the catching-up algorithm.  In addition, we use only one projection at each step without computing the resolvent of the Lipschitz continuous mapping $f(\Id -v)^{-1}$, which keeps the spirit of forward-backward technique.\\
\end{remark}
The estimation  in the proof of Theorem \ref{ro1} is indeed applied for the operator $f(\Id -v)^{-1}$. The following result provides a better estimation by exploiting the strong monotonicity of $(f, \Id -v)$ directly. It is easy to see that $\rho$ in Theorem \ref{ro1im} does not depend on $\tilde{l}$ and is strictly smaller than $\kappa$ in  Theorem \ref{ro1} if $\tilde{l}(1+l)^2\ge 1$, which happens usually in practice. 
\begin{theorem}\label{ro1im}
Suppose that  Assumptions 1, 2, 3 hold. Let $(x_n)$ be the sequence generated by Algorithm 1. Then we have 
$$
\Vert x_n-x^*\Vert\le  {\tilde{l}(1+l)} \rho^n\Vert x_0-x^*\Vert,
$$
where $\rho:=\sqrt{1-\frac{\gamma^2}{L^2(1+l)^2}}$.
\end{theorem}

\begin{proof}
Let $y_n:=(\Id -v)x_n$, $y^*:=(\Id -v)x^*$ then we have $y_{n+1}={\rm proj}_{C}(y_n-hf(x_n))$ and $y^*={\rm proj}_{C}((y^*-hf(x^*))$. The strong monotonicity of $(f,\Id -v)$ deduces that
\beq
\langle y_n-y^*,f(x_n)-f(x^*)\rangle \ge \gamma \Vert x_n-x^*\Vert^2.
\eeq
By choosing $h=\gamma/L^2$, one has
\baqn
\Vert y_{n+1}-y^*\Vert^2&\le& \Vert y_n-y^*-h(f(x_n)-f(x^*))\Vert^2\\
&\le& \Vert y_n-y^*\Vert^2-2\gamma h \Vert x_n-x^*\Vert^2+h^2L^2\Vert x_n-x^*\Vert^2\\
&=& \Vert y_n-y^*\Vert^2-\frac{\gamma^2}{L^2} \Vert x_n-x^*\Vert^2\\
&\le& (1-\frac{\gamma^2}{L^2(1+l)^2})\Vert y_n-y^*\Vert^2
\eaqn
and the conclusion follows. 
\end{proof}
\subsection{The  (strong, pseudo) monotonicity case}
If the strong monotonicity of $(f,\Id -v)$ is reduced to  the monotonicity (or  pseudo-monotonicity), we have the monotonicity (pseudo-monotonicity respectively) of the operator $f(\Id -v)^{-1}$.  Then one can apply  Tseng  algorithm (Algorithm 2) to the VI (\ref{vire}) to obtain the  weak convergence of the iterated sequence $(x_n)$  to some solution of the QVI (\ref{qvi}) \cite{Tseng} or some modified Tseng-type algorithms (see, e.g., \cite{Cai}) to have the strong convergence. In the case we have the strong pseudo-monotonicity of $(f,\Id -v)$, equivalently of $f(\Id -v)^{-1}$ , one can use the relaxed inertial projection algorithm proposed in \cite{Phan} to obtain the linear convergence. \\

\noindent \textbf{Algorithm \;2}: Given $x_0\in H$, for each $n\ge 0$, we compute 
\beq\label{direct2}
\left\{
\begin{array}{l}
   y_n=(\Id -v)x_n\\ \\
  z_{n}=(\Id -v)^{-1}{\rm proj}_{C}\Big(y_n-hf(x_n)\Big)\\ \\
   x_{n+1}=(\Id -v)^{-1}\{y_n+h(f(x_n)-f(z_n))\}.
\end{array}\right.
\eeq
\section{Application}
The technique developed in the previous sections can be used to propose a derivative-free algorithm to find a solution of the equation $f(x)=0$ where $f$ is only {assumed to be Lipschitz continuous},  by finding a Lipschitz continuous mapping $w$  such that $(f,w)$ is strongly monotone. 
{The subsequent algorithm is derived from Algorithm 1, with the choice of $w$ being $\Id-v$ and $C=H$.}
\vskip 5mm
\noindent \textbf{Algorithm \;3}: Given $x_0\in H$, for $n\ge 0$, we compute 
\beq
x_{n+1}=w^{-1}\Big(w(x_n)-hf(x_n)\Big).
\eeq
\begin{theorem}
Let $f: H \to H$ be a $L_f$-Lipschitz continuous function. Suppose that we can find a $L_w$-Lipschitz continuous mapping $w: H\to H$ such that $(f,w)$ is $\gamma$-strongly monotone and  $w^{-1}$ is well-defined single-valued. Then $f(x)=0$ has a unique solution $x^*$ and  {the sequence $(x_n)$} generated by Algorithm 3 converges to $x^*$ with linear rate. 
\end{theorem}
\begin{proof} The strong monotonicity of $(f,w)$ implies the strong monotonicity of $f w^{-1}$. Thus there exists uniquely $y^*$ such that $fw^{-1}(y^*)=0$. Thus $x^*=w^{-1}(y^*)$ is the unique solution of the equation $f(x)=0$.\\

Let $y_n:=w(x_n)$ then we have $y_{n+1}=y_n-hf(x_n)$. Since $(f,w)$ is strongly monotone, we have 
\beq
\langle y_n-y^*,f(x_n)-f(x^*)\rangle \ge \gamma \Vert x_n-x^*\Vert^2.
\eeq
By choosing $h=\gamma/L_f^2$, one has
\baqn
\Vert y_{n+1}-y^*\Vert^2&\le& \Vert y_n-y^*-h(f(x_n)-f(x^*))\Vert^2\\
&\le& \Vert y_n-y^*\Vert^2-2\gamma h \Vert x_n-x^*\Vert^2+h^2L_f^2\Vert x_n-x^*\Vert^2\\
&=& \Vert y_n-y^*\Vert^2-\frac{\gamma^2}{L_f^2} \Vert x_n-x^*\Vert^2\\
&\le& (1-\frac{\gamma^2}{L_f^2L_g^2})\Vert y_n-y^*\Vert^2
\eaqn
and one has the conclusion. 
\end{proof}
\begin{remark}\normalfont
{We usually select \( w \) as a matrix to ensure that the computation of \( w^{-1} \) can be performed with ease.
Then Algorithm 3 becomes like a quasi-Newton method: \( x_{n+1} = x_n - h w^{-1} f(x_n) \). Instead of using the Jacobian matrix of $f$ at \( x_n \), we use the fixed matrix \( w \).
}

\end{remark}
The following case provides a suitable choice of $w$.
\begin{theorem}\label{spe}
{Let \( f: \mathbb{R}^n \to \mathbb{R}^n \) be defined by \( f(x) = Ax + g(x) \), where \( A \) is an invertible matrix and \( g: \mathbb{R}^n \to \mathbb{R}^n \) is a function that is \( L_g \)-Lipschitz continuous. If \( \alpha:=\Vert A^{-1} \Vert L_g  < 1 \), then, for \( w = A \), the sequence \( \{x_k\} \) produced by Algorithm 3 is guaranteed to converge to a zero \( x^* \) of \( f \), and this convergence is at a linear rate.}
\end{theorem}
\begin{proof}
{Starting from the iteration formula given by Algorithm 3, we have
\begin{align*}
x_{n+1} &= x_{n} - hA^{-1}f(x_n) \\
&= x_{n} - hA^{-1}(Ax_n + g(x_n)) \\
&= (1 - h)x_n - hA^{-1}g(x_n),
\end{align*}
Now, let \( x^* \) be a zero of $f$, i.e. \( f(x^*) = Ax^* + g(x^*) = 0 \). We have
\begin{align*}
\Vert x_{n+1} - x^* \Vert &= \Vert (1 - h)(x_n - x^*) - hA^{-1}(g(x_n) - g(x^*)) \Vert \\
&\leq (1 - h)\Vert x_n - x^* \Vert + h\Vert A^{-1} \Vert \Vert g(x_n) - g(x^*) \Vert \\
&\leq (1 - h)\Vert x_n - x^* \Vert + h\alpha \Vert x_n - x^* \Vert \\
&= (1 - h + h\alpha)\Vert x_n - x^* \Vert \\
&= (1 - h(1 - \alpha))\Vert x_n - x^* \Vert.
\end{align*}
Since \( \alpha < 1 \), it follows that \( 1 - h(1 - \alpha) < 1 \), and hence \( \{x_n\} \) converges to \( x^* \) at a linear rate. This completes the proof.
}
\end{proof}
\begin{remark}\normalfont
{
Under the assumptions of Theorem \ref{spe}, the pair \( (f,A) \) is strongly monotone. In fact, for any vectors \( x_1, x_2 \), the following inequality holds:
\begin{align*}
\langle A(x_1 - x_2) + g(x_1) - g(x_2), A(x_1 - x_2) \rangle &= \Vert A(x_1 - x_2) \Vert^2 + \langle g(x_1) - g(x_2), A(x_1 - x_2) \rangle \\
&\geq (1 - \alpha) \Vert A(x_1 - x_2) \Vert^2 \\
&\geq (1 - \alpha) c \Vert x_1 - x_2 \Vert^2,
\end{align*}
where \( c \) represents the smallest positive eigenvalue of \( A^TA \).
}
\end{remark}

\section{Numerical Examples}
{The aim of this section is to {illustrate} the theoretical parts of the last sections on some simple examples in $\R^2$ and $\R^3$. It is important to emphasize that these examples are purely intended for illustrative purposes.}

\begin{example}\normalfont
{Let $f:\R^2\to\R^2,\;x=(x_1,x_2)\mapsto f(x)$ defined by}
$$f(x)=\left( \begin{array}{ccc}
3 &\;\; 1\\ \\
1 &\;\; 4  
\end{array} \right)x+\left( \begin{array}{ccc}
0.5\cos^3(x_2)  \\ \\
0.7\sin(x_1)  
\end{array} \right), v=\left( \begin{array}{ccc}
-0.2 &\;\; -0.4\\ \\
-0.4 &\;\; -0.6
\end{array} \right),$$
and $C=[-30,40]^2$. 
\begin{figure}
\begin{center}
\includegraphics[width=3.5in]{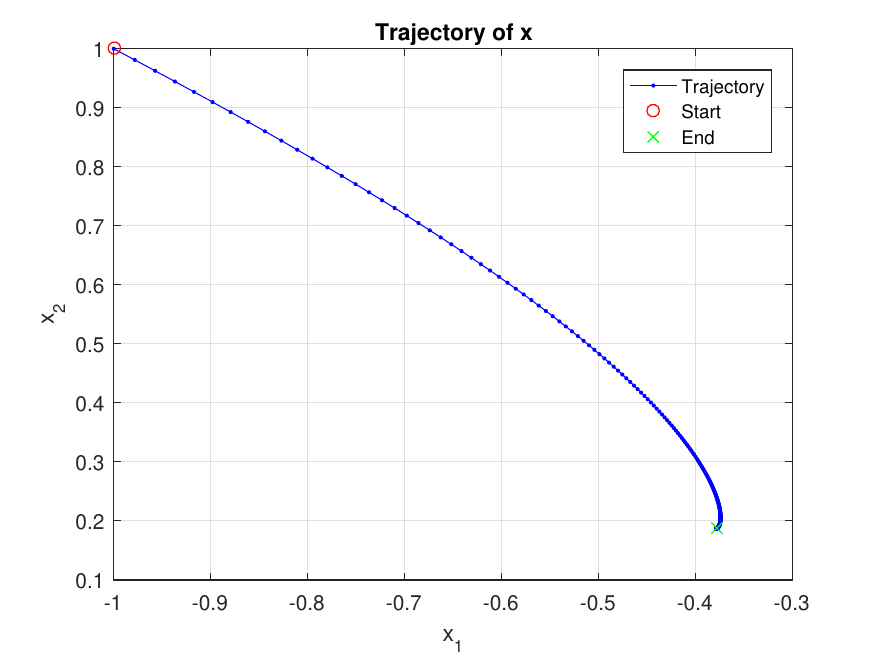}\includegraphics[width=3.5in]{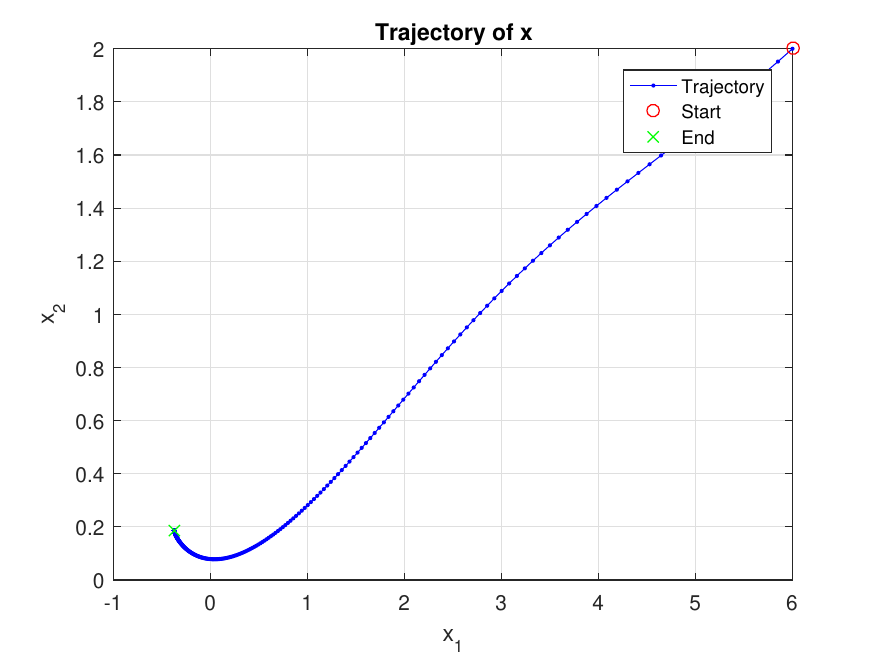}
\caption{Phase portrait for Example 1: Trajectories originating from two distinct initial conditions}
\label{fig3}
\end{center}
\end{figure}
{Then, the function \( f \) is \( 1.68 \)-strongly monotone and \( 5.32 \)-Lipschitz continuous, while \( v \) is \( 0.85 \)-Lipschitz continuous. We first note that all the conditions of Theorem~\ref{ro1} are satisfied. However, the relation \( 0.85 > \frac{1.68}{5.32} \) implies that classical approaches are not applicable.} Figure~\ref{fig3} presents the phase portrait for Example 1, depicting the trajectories resulting from two different initial conditions, \( x_0 = [6; 2] \) and \( x_0 = [-1; 1] \), obtained by implementing Algorithm 1 with a step size of \( h = 0.01 \). These trajectories demonstrate the evolution of the system across iterations, culminating in the convergence towards the equilibrium point \( x = [-0.3785; 0.1870] \).

\end{example}
\begin{example}\normalfont
{Let $f:\R^3\to\R^3,\;x=(x_1,x_2,x_3)\mapsto f(x)$ defined by}
$$f(x)=\left( \begin{array}{ccc}
5 &\;\; 7&\;\; 2 \\ \\
4 &\;\; 3 &\;\;  -3 \\ \\
8&\;\;1 &\;\;2
\end{array} \right)x+\left( \begin{array}{ccc}
1.2\vert\sin^3(x_2) \vert \\ \\
1.1\vert\sin(x_3)\vert  \\ \\
\cos^3(\vert x_1\vert+x_3)
\end{array} \right), v=\left( \begin{array}{ccc}
-9 &\;\; -14&\;\; -4 \\ \\
-8 &\;\; -5 &\;\;  6 \\ \\
-16&\;\;-2 &\;\;-3
\end{array} \right),$$
and $C=[-400,500]^3$. 
{We observe that the function $f$ lacks monotonicity and  $v$ is  Lipschitz continuous with a constant of $23.12$ while the pair $(f, \Id - v)$ demonstrates strong monotonicity. \\
By employing Algorithm 1 with an initial value of $x_0 = [43; 22 ;55]$ and a step size of $h=0.3$, we successfully obtain an approximation of the unique solution after a total of $83$ iterations. This solution is given by $x^* =  [-0.1249; 0.1025; -0.0469]$. Meanwhile, the catching-up algorithm proves unsuccessful in this case.}
\end{example}
\begin{example}\normalfont
{Let $f:\R^3\to\R^3,\;x=(x_1,x_2,x_3)\mapsto f(x)$ and Let $v:\R^3\to\R^3,\;x=(x_1,x_2,x_3)\mapsto v(x)$ defined by}
 $$f(x)=\left( \begin{array}{ccc}
5 &\;\; 7&\;\; 2 \\ \\
4 &\;\; 3 &\;\;  -3 \\ \\
8&\;\;1 &\;\;2
\end{array} \right)x+\left( \begin{array}{ccc}
0.8\sin^2(x_2)  \\ \\
0.7\sin(x_3)  \\ \\
0.8\cos^3(x_1+x_3)
\end{array} \right)$$
$$
v(x)=Ax+v_1(x)=\left( \begin{array}{ccc}
-9 &\;\; -14&\;\; -4 \\ \\
-8 &\;\; -5 &\;\;  6 \\ \\
-16&\;\;-2 &\;\;-3
\end{array} \right)x+\left( \begin{array}{ccc}
0.6\cos^2(x_2)  \\ \\
0.5\sin(x_1)  \\ \\
0.7\sin^2(x_3)
\end{array} \right).
$$
{It's worth noting that $(f, \Id - v)$ exhibits strong monotonicity. When we apply Algorithm 1 with an initial value of $x_0 = [5;4 ;2]$ and a step size of $h=0.3$, we reach the unique solution after a total of  110 iterations. This numerical solution is given by $x^* = [-0.0868; 0.6040; 0.6839]$.}
\end{example}

\begin{example}\normalfont
{
Consider the problem of solving the equation \( f(x) = 0 \), where
$$f(x)=\left( \begin{array}{ccc}
5 &\;\; 7&\;\; 2 \\ \\
4 &\;\; 3 &\;\;  -3 \\ \\
8&\;\;1 &\;\;2
\end{array} \right)x+\left( \begin{array}{ccc}
3\vert\sin(x_3)\vert  \\ \\
\vert\cos(x_1)\vert+2\vert \sin(x_2)\vert  \\ \\
3\cos(x_1+\vert x_3\vert)
\end{array} \right)=Ax+g(x).$$
Using the initial point \( x_0 = [10^4, 2 \times 10^4, 3 \times 10^4] \), the application of Algorithm 3 to the function \( f \), with \( v = \Id - A \) and \( h = 1 \), successfully converged to the solution \( {x^*} = [-0.0931; 0.0816; -0.0555] \) in 18 iterations. The function evaluated at \( {x^*} \) is \( f({x^*}) = 10^{-13} \times [-0.0003; 0.2991; 0.0877] \), which is sufficiently close to zero, confirming that \( {x^*} \) is indeed an approximate solution to the equation \( f(x) = 0 \) within the specified numerical tolerance. This approach is beneficial because it does not require calculating or estimating the Jacobian of $f$, which is not possible for \( f \) as it is a nonsmooth function. 
}

\end{example}
\section{Conclusions} 
In the paper, we introduce the monotonicity of a pair of functions which is essential  in the study of asymptotic behavior of state-dependent SWPs and the convergence analysis in QVIs.  As a product, we propose a new simple algorithm for QVIs with linear convergence rate under the strong monotonicity of the involved pair of functions, which is remarkably weaker than the classical conditions. The idea can be used to consider the cases of monotonicity, (strong) pseudo-monotonicity of QVIs which has not been studied in the literature. We also propose a derivative-free algorithm to find a zero of a nonsmooth Lipschitz continuous mapping. The presence of the inverse operator $(\Id-v)^{-1}$ broaden the applications. 
{It's desirable to find simpler algorithms that don't require computing the inverse operator, yet still maintain the same convergence condition and rate. This is a pertinent open question. While we've tested our approach on academic examples, it would be particularly interesting to apply these methods to more substantive concrete examples that fit the QVI framework.
These investigations open up new paths for upcoming research.
}

%
%

\end{document}